\documentclass[oneside,british, reqno]{amsart}
\usepackage[T1]{fontenc}
\usepackage[latin9]{inputenc}
\usepackage{color}
\usepackage{babel}
\usepackage{units}
\usepackage{mathtools}
\usepackage{amsthm}
\usepackage{bm}
\usepackage{amssymb}
\usepackage{verbatim}
\usepackage{amsfonts,amsmath,latexsym,amssymb,amsthm}
\usepackage[unicode=true,pdfusetitle,
 bookmarks=true,bookmarksnumbered=false,bookmarksopen=false,
 breaklinks=false,pdfborder={0 0 0},backref=false,colorlinks=false]
 {hyperref}

\makeatletter
\numberwithin{equation}{section}
\numberwithin{figure}{section}
\theoremstyle{plain}
\newtheorem{thm}{\protect\theoremname}
  \theoremstyle{remark}
  \newtheorem{rem}{\protect\remarkname}
  \theoremstyle{plain}
  \newtheorem{cor}{\protect\corollaryname}
 \theoremstyle{definition}
 \newtheorem*{defn*}{\protect\definitionname}
  \theoremstyle{plain}
  \newtheorem{lem}{\protect\lemmaname}
    \newtheorem{prop}{\protect\propositionname}

\newtheorem{example}[thm]{\protect\examplename}
\usepackage{babel}

\def\locphi{{\iota}}
\def\IR{\mathbb R}
\def\IQ{\mathbb Q}
\def\IZ{\mathbb Z}

\def\Sm{S}
\def\Rm{R}

\def\GL{\text{GL}}
\def\SL{\text{SL}}
\def\s{s}

  \addto\captionsamerican{\renewcommand{\definitionname}{Definition}}
  \addto\captionsamerican{\renewcommand{\lemmaname}{Lemma}}
  \addto\captionsamerican{\renewcommand{\remarkname}{Remark}}
  \addto\captionsbritish{\renewcommand{\definitionname}{Definition}}
  \addto\captionsbritish{\renewcommand{\lemmaname}{Lemma}}
  \addto\captionsbritish{\renewcommand{\remarkname}{Remark}}
  \addto\captionsenglish{\renewcommand{\definitionname}{Definition}}
  \addto\captionsenglish{\renewcommand{\lemmaname}{Lemma}}
  \addto\captionsenglish{\renewcommand{\remarkname}{Remark}}
  \providecommand{\definitionname}{Definition}
  \providecommand{\lemmaname}{Lemma}
  \providecommand{\remarkname}{Remark}
\addto\captionsamerican{\renewcommand{\corollaryname}{Corollary}}
\addto\captionsamerican{\renewcommand{\theoremname}{Theorem}}
\addto\captionsbritish{\renewcommand{\corollaryname}{Corollary}}
\addto\captionsbritish{\renewcommand{\theoremname}{Theorem}}
\addto\captionsenglish{\renewcommand{\corollaryname}{Corollary}}
\addto\captionsenglish{\renewcommand{\theoremname}{Theorem}}
\providecommand{\corollaryname}{Corollary}
\providecommand{\theoremname}{Theorem}

\makeatother

  \addto\captionsamerican{\renewcommand{\definitionname}{Definition}}
  \addto\captionsamerican{\renewcommand{\lemmaname}{Lemma}}
  \addto\captionsamerican{\renewcommand{\remarkname}{Remark}}
  \addto\captionsbritish{\renewcommand{\definitionname}{Definition}}
  \addto\captionsbritish{\renewcommand{\lemmaname}{Lemma}}
  \addto\captionsbritish{\renewcommand{\remarkname}{Remark}}
  \addto\captionsenglish{\renewcommand{\definitionname}{Definition}}
  \addto\captionsenglish{\renewcommand{\lemmaname}{Lemma}}
  \addto\captionsenglish{\renewcommand{\remarkname}{Remark}}
  \providecommand{\definitionname}{Definition}
  \providecommand{\lemmaname}{Lemma}
  \providecommand{\remarkname}{Remark}
\addto\captionsamerican{\renewcommand{\corollaryname}{Corollary}}
\addto\captionsamerican{\renewcommand{\theoremname}{Theorem}}
\addto\captionsbritish{\renewcommand{\corollaryname}{Corollary}}
\addto\captionsbritish{\renewcommand{\theoremname}{Theorem}}
\addto\captionsenglish{\renewcommand{\corollaryname}{Corollary}}
\addto\captionsenglish{\renewcommand{\theoremname}{Theorem}}
\providecommand{\corollaryname}{Corollary}
\providecommand{\theoremname}{Theorem}
\providecommand{\propositionname}{Proposition}
\providecommand{\examplename}{Example}
\begin{document}
\selectlanguage{british}%
\global\long\def\mod{\,\mathrm{mod\,\,}}
 \global\long\def\Nm{\mathrm{Nm\,}}

\title{On a Counting Theorem of Skriganov}
\begin{abstract}
We prove a counting theorem concerning the number of lattice points for the dual lattices of
weakly admissible lattices in an inhomogeneously expanding box, which generalises a counting theorem of Skriganov. 
The error term is expressed in terms of a certain function
$\nu(\Gamma^\perp,\cdot)$ of the dual lattice $\Gamma^\perp$, and we carefully
analyse the relation of this quantity with $\nu(\Gamma,\cdot)$.
In particular, we show that $\nu(\Gamma^\perp,\cdot)=\nu(\Gamma,\cdot)$
for any unimodular lattice of rank 2, but that for higher ranks it is in general
not possible to bound one function in terms of the other. 
Finally, we apply our counting theorem to establish asymptotics
for the number of Diophantine approximations with bounded denominator
as the denominator bound gets large. 
\end{abstract}

\author{Niclas Technau\and{}Martin Widmer}

\thanks{The first author was supported by the Austrian Science Fund (FWF):
W1230 Doctoral Program ``Discrete Mathematics''.}

\address{Institut für Analysis und Zahlentheorie\\
 Technische Universität Graz\\
 Steyrergasse 30\\
 A-8010\\
 technau@math.tugraz.at}

\address{Department of Mathematics\\
 Royal Holloway\\
 University of London\\
 TW20 0EX Egham\\
 UK\\ 
 Martin.Widmer@rhul.ac.uk}
 

\maketitle

\section{Introduction}
In the present article, we are mainly concerned with four objectives. 
Firstly, we prove an explicit version of Skriganov's celebrated
counting result \cite[Thm. 6.1]{Skriganov} for lattice points 
of unimodular weakly admissible lattices in homogeneously expanding aligned boxes. 
Secondly, we use this version to generalise Skriganov's theorem to inhomogeneously expanding,
aligned boxes. Thirdly, we carefully investigate the relation between $\nu(\Gamma,\cdot)$
(see (\ref{mufunction}) for the definition) and $\nu(\Gamma^\perp,\cdot)$ 
of the dual lattice $\Gamma^\perp$ which captures the dependency 
on the lattice in these error terms. And fourthly, we apply our counting result
to count Diophantine approximations.

To state our first result, we need to introduce some notation.
By writing $f\ll g$ (or $f \gg g$) for functions $f,g$, we mean that there is a constant $c>0$ 
such that $f(x) \leq c g(x)$ (or $cf(x) \geq g(x)$) holds for all admissible values of $x$;
if the implied constant depends on certain parameters, then this dependency
will be indicated by an appropriate subscript. 
Let $\Gamma \subseteq\IR^{n}$ be a unimodular lattice, 
and let $\Gamma^{\perp}\coloneqq \left\{ w\in\IR^{n}:
\,\langle v,w\rangle\in\mathbb{Z} \quad\forall_{v\in\Gamma}\right\}$ 
be its dual lattice with respect to the standard inner product $\langle\cdot,\cdot\rangle$.
Let $\gamma_n$ denote the Hermite constant, and for $\rho > \gamma_{n}^{\nicefrac{1}{2}}$ set
\begin{alignat}1\label{mufunction}
\nu(\Gamma,\rho)\coloneqq\min\left\{ \vert x_{1}\cdots x_{n}\vert 
:\,\,x\coloneqq(x_{1}, \ldots, x_{n})^{T}\in\Gamma,\,
0<\left\Vert x\right\Vert_{2} < \rho\right\} 
\end{alignat}
where $\left\Vert \cdot\right\Vert _{2}$ denotes the Euclidean norm.
We say $\Gamma$ is weakly admissible if  $\nu(\Gamma,\rho)>0$ 
for all $\rho > \gamma_{n}^{\nicefrac{1}{2}}$.
Note that this happens if and only if $\Gamma$ has trivial intersection with every coordinate subspace.

Furthermore, let $\mathcal{T}\coloneqq \mathrm{diag}(t_{1},\ldots,t_{n})$ for $t_{i}>0$ 
be the diagonal matrix with diagonal entries $t_{1},\ldots,t_{n}$, and let $y\in \IR^n$.
We set 
$$B\coloneqq \mathcal{T} \left[0,1\right]^{n}+ y,$$ 
and we call such a set  an aligned box.
Moreover, we define 
$$T \coloneqq(\det \mathcal{T})^{\nicefrac{1}{n}}
\cdot \Vert \mathcal{T}^{-1} \Vert_{2}
=\frac{(t_1\cdots t_n)^{\nicefrac{1}{n}}}{\min\{t_1,\ldots,t_n\}}\geq 1$$ 
where $\Vert\cdot \Vert_{2}$ denotes the operator norm induced by the Euclidean norm.
Then, our generalisation of Skriganov's theorem reads as follows.
\begin{thm}\label{thm: inhomogenous Skriganov}
	Let  $n\geq 2$, let $\Gamma\subseteq \IR^{n}$ be a unimodular lattice, 
	and let $B\subseteq \IR^{n}$ be as above.
	Suppose $	\Gamma^{\perp}$ is weakly admissible, and 
	$\rho>\gamma_{n}^{\nicefrac{1}{2}}$. Then,
	\begin{equation}\label{eq: inhomogenous Skriganov bound}
		\left|\#(\Gamma\cap B)-\mathrm{vol}(B)\right|
		\underset{n}{\ll}\frac{1}{\nu\bigl(\Gamma^{\perp}, T^{\star} \bigr)}
		\biggl(
			\frac{(\mathrm{vol}(B))^{1-\nicefrac{1}{n}}}{\sqrt{\rho}}+
			\frac{R^{n-1}}{\nu(\Gamma^{\perp},2^{R}T)}
		\biggr)
	\end{equation}
	where $x^{\star}:=\max \left\{ \gamma_{n},x\right\}$, and
	$R\coloneqq n^{2}+\log\frac{\rho^{n}}{\nu(\Gamma^{\perp},\rho T)}$. 
\end{thm}
Note that $\rho^{n}/\nu(\Gamma^{\perp},\rho)\geq n^{n/2}$ 
by the inequality between arithmetic and geometric mean. Since $T\geq 1$ and 
\begin{alignat}1\label{Hermite}
\gamma_n\leq ({4}/{3})^{(n-1)/2},
\end{alignat}
we have $(2^{R}T)^{\star}=2^{R}T$, and hence, the far right hand-side
in (\ref{eq: inhomogenous Skriganov bound}) is well-defined.

The lattice $\Gamma$ is called admissible
if $\Nm(\Gamma):=\lim_{\rho\rightarrow \infty}\nu(\Gamma,\rho)>0$.
It is easy to show that if $\Gamma$ is admissible then also $\Gamma^\perp$
is admissible (see \cite[Lemma 3.1]{Skriganov1994}).
In this case we can choose $\rho=(\mathrm{vol}B)^{2-2/n}$,
provided the latter is greater than $\gamma_{n}^{\nicefrac{1}{2}}$, 
to recover the following impressive result of Skriganov (\cite[Theorem 1.1 (1.11)]{Skriganov1994})
	\begin{equation}\label{eq:admissible Skriganov bound}
		\left|\#(\Gamma\cap B)-\mathrm{vol}(B)\right|
		\underset{n, \Nm(\Gamma^\perp)}{\ll}(\log(\mathrm{vol}(B))^{n-1}.
	\end{equation}However, if $\Gamma$ is only weakly admissible,
then it can happen that $\Gamma^\perp$ is not weakly admissible;
see Example \ref{ex: Gamma weakly adm. doesn t imply Gamma_perp weakly adm}.
But this is a rather special situation and typically, e.g.,
if the entries of $A$ are algebraically independent, 
see Lemma \ref{cor: weak admissib. of dual of algb. independet lattice},
then $\Gamma=A\IZ^n$ and its dual are both weakly admissible. 
This raises the question whether, or under which conditions,
one can control $\nu(\Gamma^\perp,\cdot)$  by $\nu(\Gamma,\cdot)$.
We have the following result where we use the convention that for an integral domain $R$ 
the group of all matrices in $R^{n\times n}$ 
with inverse in $R^{n\times n}$ is denoted by $\mathrm{GL}_{n}(R)$. 
\begin{prop}\label{lem: lattices where nu fncs. equal}
	Let $\Gamma=A\IZ^n$, and suppose there exist $\Sm, \Rm$ both in $ \GL_n(\IZ)$ such that
	\[A^{T}\Sm A=\Rm,\]
	and suppose $\Sm$ has exactly one non-zero entry in each column and in each row. Then, we have 
	\begin{equation}\label{eq: nu of dual lattice vs. nu of lattice}
		\nu(\Gamma^{\perp},\cdot)=\nu(\Gamma,\cdot).		
	\end{equation}
\end{prop}
A special case of Proposition \ref{lem: lattices where nu fncs. equal} 
shows that $\nu(\Gamma^\perp,\cdot)=\nu(\Gamma,\cdot)$
whenever $\Gamma=A\IZ^n$ with a symplectic matrix $A$, in particular, 
whenever\footnote{Let us write $Sp_{2m}(\IR)$ for the symplectic subgroup of
$GL_{2m}(\IR)$ and $SL_n(\IR)$ for the special linear subgroup
of $GL_n(\IR)$. The fact $Sp_2(\IR)=SL_2(\IR)$ can be checked directly.} 
$\Gamma$ is a unimodular lattice in $\IR^2$.
In these cases, one can directly compare Theorem \ref{thm: inhomogenous Skriganov}
with a recent result \cite[Theorem 1.1]{weaklyadmissible} of the second author,
and we refer to \cite{weaklyadmissible} for more on that. 
On the other hand, our next result shows that in general $\nu(\Gamma,\cdot)$ 
can decay arbitrarily quickly even if we control $\nu(\Gamma^\perp,\cdot)$.
\begin{thm} \label{prop: changing nu to nu dual }
	Let $n\geq3$, and let $\psi:(0,\infty)\rightarrow(0,1)$
	be non-increasing. Then, there exists a unimodular,
	weakly admissible lattice $\Gamma\subseteq\IR^{n}$, 
	and a sequence $\{\rho_{l}\}\subseteq (\gamma_{n}^{\nicefrac{1}{2}},\infty)$
	tending to $\infty$, as $l\rightarrow\infty$, such that 
	\[
	\nu(\Gamma^\perp,\rho)\gg\rho^{-n^{2}},
	\]
	and 
	\[
	\nu(\Gamma,\rho_{l})\leq\psi(\rho_{l})
	\]
	for all $l\in\mathbb{N}=\{1,2,3,\ldots\}$ and for all $\rho>\gamma_n^{\nicefrac{1}{2}}$.
\end{thm}
In the case where exactly one of the functions $\nu(\Gamma,\cdot)$, and $\nu(\Gamma^\perp,\cdot)$ 
is controllable while the other one decays very quickly
either Theorem \ref{thm: inhomogenous Skriganov}
or \cite[Theorem 1.1]{weaklyadmissible}  provides a reasonable error term,
but certainly not both. This highlights the complementary aspects 
of Theorem \ref{thm: inhomogenous Skriganov}, and  \cite[Theorem 1.1]{weaklyadmissible}. 
Theorem \ref{prop: changing nu to nu dual } is deeper 
than Proposition \ref{lem: lattices where nu fncs. equal}, and relies on a recent result
of Beresnevich about the distribution of badly approximable vectors on manifolds.

Next, we apply Theorem \ref{thm: inhomogenous Skriganov}
to deduce counting results for Diophantine approximations.
We start with a bit of historical background on this, and related problems.
Let $\alpha\in\IR$, let $\locphi:[1,\infty)\rightarrow(0,1]$ be
a positive decreasing function, and let $N_{\alpha}^{loc}(\locphi,t)$
be the number of integer pairs $(p,q)$ satisfying $|p+q\alpha|<\locphi(q),\;1\leq q\leq t$.
In a series of papers, starting in 1959, Erd\H{o}s \cite{Erdos59},
Schmidt \cite{SchmidtMTDA,SchmidtRB}, Lang \cite{AdamsLang65,LangQA,Lang66},
Adams \cite{Adams66,Adams67,Adams67II,Adams68,Adams68II,Adams69II,Adams69,Adams71},
Sweet \cite{Sweet73}, and others, considered the problem of finding
the asymptotics for $N_{\alpha}^{loc}(\locphi,t)$ as $t$ gets large.

Schmidt \cite{SchmidtMTDA} has shown that for almost every\footnote{Here ``almost every'' refers 
always to the Lebesgue measure.} $\alpha\in\IR$ the asymptotics 
are given by the volume of the corresponding
subset of $\IR^{2}$, provided the latter tends to infinity. 
This is false for quadratic $\alpha$;
there with $\locphi(q)=1/q$ the volume is $2\log(t)+O(1)$, and by
Lang's result $N_{\alpha}^{loc}(1/q,t)\sim c_{\alpha}\log(t)$
but Adams \cite{Adams68II} has shown that $c_{\alpha}\neq2$.

Opposed to the above ``localised'' setting, where the bound on $|p+q\alpha|$
is expressed as a function of $q$, we consider the ``non-localised''
(sometimes called ``uniform'') situation, where the bound is expressed
as a function of $t$. Furthermore, we shall consider the more general
asymmetric inhomogeneous setting. Let $\alpha\in(0,1)$ be irrational,
$\varepsilon,t\in(0,\infty)$, and let $y\in\IR$. We define
the counting function 
\begin{equation}
N_{\alpha,y}(\varepsilon,t)=\#\biggl\{(p,q)\in \mathbb{Z}\times\mathbb{N}:\begin{array}{c}
0\leq p+q\alpha -y\leq\varepsilon,\\
0\leq q\leq t
\end{array}\biggr\}.
\label{eq: definition of the counting function}
\end{equation}

If the underlying set is not too stretched, then $N_{\alpha,y}(\varepsilon,t)$
is roughly the volume $\varepsilon t$ of the set in which we are
counting lattice points. If we let $\varepsilon=\varepsilon(t)$ be a function
of $t$ with $t=o(t\varepsilon)$ we have, by simple standard estimates,
\begin{alignat}{1}
N_{\alpha,y}(\varepsilon,t)\sim\varepsilon t\label{Asy}
\end{alignat}
for any pair $(\alpha,y)\in((0,1)\setminus\mathbb{Q})\times\IR$ whatsoever.
To get non-trivial estimates for our counting function, we need information on the Diophantine
properties of $\alpha$. Let $\phi:\left(0,\infty\right)\rightarrow(0,1)$ be 
a non-increasing function such that
\begin{equation}
q\bigl|p+q\alpha\bigr|\geq\phi\left(q\right)\label{eq: definition of phi as lower bound}
\end{equation}
holds for all $(p,q)\in\mathbb{Z}\times\mathbb{N}$.
Then \cite[Theorem 1.1]{weaklyadmissible} implies that 
\begin{alignat}{1}
|N_{\alpha,y}(\varepsilon,t)-\varepsilon t|
\ll_{\alpha}\sqrt{\frac{\varepsilon t}{\phi(t)}}.\label{prevest}
\end{alignat}
Hence, unlike in the localised setting, for badly approximable $\alpha$ the asymptotics are given
by the volume as long as the volume tends to infinity. 

Our next result significantly improves the error term in (\ref{prevest}), provided $\alpha$
is ``sufficiently'' badly approximable, i.e., provided $\phi(t)$
decays slowly enough. 
We assume that
\begin{equation}
\varepsilon t>4\qquad\mathrm{and}\qquad0<\varepsilon<\sqrt{\alpha}.
\label{eq: assumptions on varepsilon t and alpha} 
\end{equation}
\begin{cor}\label{thm: main counting bound}
Put $E:=\frac{\varepsilon t}{\phi(4 t \sqrt{\varepsilon t})}$,	
and $E':=168\sqrt{\varepsilon t^{3}}E$.
Then, we have 
	\begin{equation}
		\left|
			N_{\alpha,y}(\varepsilon,t)-\varepsilon t
		\right| 
		\underset{\alpha}{\ll}
		\frac{\log E}{\phi^{2}(E')}.\label{eq: main estimate}
	\end{equation}
\end{cor}
In particular, 
if $\alpha$ is badly approximable then 
\begin{equation}
\left|N_{\alpha,y}(\varepsilon,t)-\varepsilon t\right|\underset{\alpha}{\ll}\log(\varepsilon t).
\end{equation}

\section{An explicit version of Skriganov's counting theorem}\label{Section: Skriganov's Counting Theorem}

Let $\Gamma\subseteq \IR^n$ be a lattice, and let $\lambda_{i}(\Gamma)$ 
denote the $i$-th successive minimum 
	of $\Gamma$ with respect to the Euclidean norm ($1\leq i\leq n$).
For $r>0$ we introduce a special
	set of diagonal matrices 
	\[
	\Delta_{r}:=\left\{ \delta\coloneqq 
	\mathrm{diag}(2^{m_{1}},\ldots,2^{m_{n}}):\,m=(m_{1},\ldots,m_{n})^{T}\in
	\mathbb{Z}^{n},\,\left\Vert m\right\Vert _{2}<r,\, \det \delta=1 \right\}, 
	\]
	and we put 
	\[
	S(\Gamma,r):=\sum_{\delta\in\Delta_{r}}(\lambda_{1}(\delta\Gamma))^{-n}.
	\]
Now we can state Skriganov's result. In fact, his result is more general,
and applies to any convex, compact polyhedron. On the other hand,
the dependency on $B$ and $\Gamma$ in the error term is not explicitly
stated in his counting result \cite[Thm. 6.1]{Skriganov}.
By carefully following his reasoning, see Remark \ref{rem: Skriganovs reasoning} below, 
we find the following explicit version of his result.
Recall that $\gamma_{n}$ denotes the Hermite constant. 
 \begin{thm} \label{thm: Skriganovs Counting Theorem} [Skriganov, 1998]
	 Let $n\geq 2$ be an integer, let $\Gamma\subseteq \IR^n$ be a unimodular lattice,
	 and let $B\subseteq \IR^n$ be an aligned box of volume $1$.
	 Suppose $\Gamma^{\perp}$ is weakly admissible,
	 and $\rho>\gamma_{n}^{\nicefrac{1}{2}}$. Then, for $t>0$,
	\begin{equation}\label{eq: Skriganovs estimate}
		\left|\#(\Gamma\cap tB)-t^{n}\right|
		\underset{n}{\ll}(\left|\partial B\right|\lambda_{n}(\Gamma))^{n}
		\cdot(t^{n-1}\rho^{-\nicefrac{1}{2}}+S(\Gamma^{\perp},r))
	\end{equation}
	where $r:=n^{2}+\log\frac{\rho^{n}} {\nu(\Gamma^{\perp},\rho)}$,
	and $\left|\partial B\right|$ denotes the surface area of $B$.
\end{thm}

\begin{rem}\label{rem: Skriganovs reasoning}
	The references and notation in this remark are the same as in \cite{Skriganov}.
	Put $\mathcal{O}\coloneqq tB$, fix a mollifier $\omega$ as in (11.3),
	and denote by $\tilde{\chi}(\mathcal{O},\cdot)$ the Fourier transform
	of the characteristic function $\chi(\mathcal{O},\cdot)$ of $\mathcal{O}$.
	Skriganov applies Lemma 11.1 to the error term 
	\[
	R(\mathcal{O},\Gamma)\coloneqq\sup_{X\in\IR^{n}}\left|
	\#((\mathcal{O}+X)\cap\Gamma)-\mathrm{vol}(\mathcal{O})\right|
	\]
	to estimate it by
	\[
	R(\mathcal{O},\Gamma)\leq\mathrm{vol}(\mathcal{O}_{\tau}^{+})
	-\mathrm{vol}(\mathcal{O}_{\tau}^{-})+\sup_{X\in\IR^{n}}
	\bigl(
		\bigl|\mathcal{R}_{\tau}^{+}(\mathcal{O},X)
		\bigr|+\bigl|\mathcal{R}_{\tau}^{-}(\mathcal{O},X)\bigr|
	\bigr)
	\]
	where $\mathcal{O}_{\tau}^{\pm}$
	is a $\tau$-coapproximation\footnote{Given a compact region $\mathcal{O}\subseteq\IR^{n}$ and a
	real number $\tau>0$, compact regions $\mathcal{O}_{\tau}^{\pm}$
	are called $\tau$-coapproximations to $\mathcal{O}$,
	if $\mathcal{O}_{\tau}^{-}\subseteq\mathcal{O}\subseteq\mathcal{O}_{\tau}^{+}$
	and $\mathrm{dist}(\partial\mathcal{O},\partial\mathcal{O}_{\tau}^{\pm})\geq\tau$
	are satisfied.} of $\mathcal{O}$, and $\mathcal{R}_{\tau}^{\pm}$ are the Fourier series 
	\[
	\mathcal{R}_{\tau}^{\pm}(\mathcal{O},X)\coloneqq\sum_{\gamma\in\Gamma^{\perp}
	\setminus
	\left\{ 0\right\} }\tilde{\chi} (\mathcal{O}_{\tau}^{\pm},\gamma)\tilde{\omega}
	(\tau\gamma)e^{-2\pi i\left\langle \gamma,X	\right\rangle }
	\]
	defined in (11.5) where $\tilde{\omega}$ denotes the Fourier transform of $\omega$.
	Observe that $\vert \partial B \vert\geq 1$, and that without loss of generality
	$B$ is centred at the origin, i.e., $y=-\frac{1}{2}(t_{1},\ldots,t_{n})^T$.
	Hence, we can choose
	$\mathcal{O}_{\tau}^{\pm}\coloneqq(t\pm\left|\partial B\right|\tau)B$ with $0< \tau < 1$,
	and thus
	\[
	\mathrm{vol}(\mathcal{O}_{\tau}^{+})-\mathrm{vol}
	(\mathcal{O}_{\tau}^{-})\underset{n}		 
	{\ll}\left|\partial B\right|^{n}t^{n-1}\tau.
	\]
	As noted in (6.6), since $B$ is an aligned box, the average
	$S(\Gamma_{\mathfrak{f}},\cdot)$ simplifies to $S(\Gamma^{\perp},\cdot)$,
	and $\nu(\Gamma_{\mathfrak{f}}^{\perp},\cdot)=\nu(\Gamma^{\perp},\cdot)$ 
	for each flag of faces $\mathfrak{f}$ of $B$.\\
	Now $\mathcal{R}_{\tau}^{\pm}$ is decomposed via (12.7) into
	partial sums $\mathcal{A}_{\tau,\rho}^{\pm}$ plus remainder terms 
	$\mathcal{B}_{\tau,\rho}^{\pm}$
	which are defined in (12.8) and (12.9), respectively. 
	Let $\omega_{2}$ denote the Fourier transform of $\omega_{1}$ (cf. p. 57). Due to (12.12),
	there is a constant $c=c(\omega_{1},\omega_{2})$,
	independent of $\Gamma,t,\rho,\tau$,
	such that\footnote{Conceivably, we should mention a
	typo regarding the definition of $r_{\mathfrak{f}}$
	in (6.5): $r_{\mathfrak{f}}$ is to be taken as in (12.13). In (12.13)
	$\varkappa_{n}$ denotes $\tau_{n}$ from Lemma 10.1, which was defined in (7.4)
	as two times the diameter of the Dirichlet-Voronoi region of the
	lattice $M$ defined in (3.3). It is easy to see that $2\tau_{n}<n^{2}$.} 
	\[
	\max_{X\in\IR^{n}}\mathcal{A}_{\tau,\rho}^{\pm}(\mathcal{O},X)
	\leq cS(\Gamma^{\perp},r)
	\]
	where we may choose $r$ to be
	\[
	r \coloneqq n^2+
		\log\frac{\rho^{n}}{\nu(\Gamma^{\perp},\rho)}.
	\]
	Hence, $c$ depends in fact only on the (fixed) mollifier $\omega_{1}$.
	Furthermore, $\mathcal{B}_{\tau,\rho}^{\pm}(\mathcal{O},X)$
	is estimated in (12.14) by
	\[
	\max_{X\in\IR^{n}}
	\bigl|
		\mathcal{B}_{\tau,\rho}^{\pm}(\mathcal{O},X)
	\bigr|	
	\leq\frac{c_{A}}{2\pi}
		\left|\partial B
		\right|t^{n-1}\tau^{-A}
	\sum_{
		\underset{\left\Vert \gamma\right\Vert _{2}>\frac{1}{8}\rho
		 }
		 {
		 \gamma\in\Gamma^{\perp}}
		 }
		\left\Vert\gamma\right\Vert _{2}^{-A-1}
	\]
	where $A>n$. Note that for $R>0$
	\[
	\#\left\{ \gamma\in\Gamma^{\perp}:\,\left\Vert \gamma\right\Vert _{2}<R\right\} 
	\underset{n}{\ll}(R/\lambda_{1}(\Gamma^{\perp})+1)^{n}.
	\]
	This in turn implies that for $k\in \mathbb{N}_{0}$ we have 
	\[
	\#\left\{ \gamma\in\Gamma^{\perp}:\,2^{k}\leq\left\Vert \gamma\right\Vert _{2}<2^{k+1}
	\right\} \underset{n}{\ll}(2^{k+1}/\lambda_{1}(\Gamma^{\perp}))^{n}.
	\]
	Using dyadic summation, and Mahler's relations 
	\begin{equation}\label{eq: Mahler's relation}
		1\leq \lambda_i(\Gamma^{\perp})\lambda_{n+1-i}(\Gamma)\leq n! \qquad (i=1,\ldots,n)
	\end{equation} yields 
	\[
	\sum_{\underset{\left\Vert \gamma\right\Vert _{2}>
	\frac{1}{8}\rho}{\gamma\in\Gamma^{\perp} }}\left\Vert \gamma\right\Vert _{2}^{-A-1}
	\underset{n}{\ll}\sum_{k>\left\lfloor
	\frac{\log(8^{-1}\rho)}{\log2}\right\rfloor }
	2^{(k+1)n}\lambda_{1}^{-n}(\Gamma^{\perp})		
	\cdot2^{-(A+1)k}\underset{n}{\ll}\lambda_{n}^{n}(\Gamma)\rho^{n-A-1}.
	\]
	Hence, 
	\[
	\bigl|\mathcal{R}_{\tau}^{\pm}(\mathcal{O},X)\bigr|
	\underset{n}{\ll}cS(\Gamma^{\perp},r)
	+c_{A}\left|\partial B\right|t^{n-1}\tau^{-A}\lambda_{n}^{n}(\Gamma)\rho^{n-A-1}.
	\]
	Specialising $A\coloneqq2n-1$ implies 
	\begin{align*}
	R(\mathcal{O},\Gamma) & \underset{n}{\ll}
	\left|\partial B\right|^{n}t^{n-1}\tau 
	+ S(\Gamma^{\perp},r)
	+\left|\partial B \right|^{n} t^{n-1} \tau^{1-2n}\lambda_{n}^{n}(\Gamma)\rho^{-n}\\
 	& \underset{n}{\ll}(\left|\partial B\right|\lambda_{n}(\Gamma))^{n}
 	(t^{n-1}\tau+ S(\Gamma^{\perp},r)+t^{n-1}\tau^{1-2n}\rho^{-n})
	\end{align*}
	where in the last inequality we used the obvious fact $|\partial B|\geq 1$.
	Finally, choosing $\tau\coloneqq\rho^{-\nicefrac{1}{2}}$ gives the
	required estimate.
\end{rem}

For proving Theorem \ref{thm: inhomogenous Skriganov}, we want to
exploit Theorem \ref{thm: Skriganovs Counting Theorem}. To this end
let $\overline{t}\coloneqq (\det\mathcal{T})^{\nicefrac{1}{n}}$, and let
\begin{equation}
U\coloneqq\overline{t}\mathcal{T}^{-1}.\label{eq: def. of U}
\end{equation} Thus,
\[
\#(\Gamma\cap B)=\#(U\Gamma\cap U(\mathcal{T}\left[0,1\right]^{n}+y))
=\#(\Lambda\cap\overline{t}(\left[0,1\right]^{n}+\mathcal{T}^{-1}(y)))
\]
where $\Lambda\coloneqq U\Gamma$. Moreover,
we conclude by Theorem \ref{thm: Skriganovs Counting Theorem} that
\begin{equation}
\left| \#(\Gamma\cap B)- \mathrm{vol}(B) \right| 
\underset{n}{\ll}\lambda_{n}^{n}(\Lambda)
\biggl(\frac{\overline{t}^{n-1}}{\sqrt{\rho}}
+S(\Lambda^{\perp},r)\biggr).\label{eq: homog. Skriganov app. to inhomog. box}
\end{equation}
For controlling the quantities on the right hand side in terms of
$\Gamma$, $\overline{t}$, $\rho$, and $\nu(\Gamma^{\perp},\cdot)$, 
we need two lemmata. 
We will frequently use the fact that if $\Gamma=A\IZ^n$ is unimodular then $\Gamma^\perp=(A^{-1})^T\IZ^n$.
As usual, we let $\mathrm{SL}_{n}(\mathbb{R})$
denote the group of all $\mathbb{R}^{n\times n}$ matrices with determinant $1$.
\begin{lem} \label{lem: estimate nu function vs. n-th successive minimum}
	Let $D\coloneqq\mathrm{diag}(d_{1},\ldots,d_{n})$ be in $\SL_{n}(\IR)$,
	and $\rho > \gamma_{n}^{\nicefrac{1}{2}}$. Then,
	\begin{equation}\label{eq: shifting opnorm D from rho to lattice in nu-function}
		\nu((D\Gamma)^{\perp},\rho)
		\geq\nu(\Gamma^{\perp},\left\Vert D\right\Vert_{2}\rho),
	\end{equation}
	and
	\begin{equation}\label{eq: first minimum vs. nu of lattice}
		\lambda_{1}^{n}(D\Gamma)\underset{n}{\gg}
		\nu(\Gamma,\left\Vert D^{-1}\right\Vert_{2}^{\star} ).
	\end{equation}
\end{lem}
\begin{proof}
	For $v\coloneqq(v_{1},\ldots,v_{n})^{T}\in\mathbb{R}^n$
	define $\mathrm{Nm}(v)\coloneqq \left| v_{1} \cdots v_{n}\right|$. We remark that 
	\begin{align*}
		\nu((D\Gamma)^{\perp},\rho) & =\nu(D^{-1}\Gamma^{\perp},\rho)\\
	 	& =\min\left\{ \mathrm{Nm}(D^{-1}v):\,v\in\Gamma^{\perp},\,
	 	0<\bigl\Vert D^{-1}v\bigr\Vert_{2} <\rho\right\} \\
		 & =\min\left\{ \mathrm{Nm}\,(v):\,v\in\Gamma^{\perp},\,0<\bigl\Vert D^{-1}v\bigr\Vert_{2}
	 	<\rho\right\}.
	\end{align*}
	If $ \Vert D^{-1} v \Vert_{2} < \rho $, then $\Vert v \Vert_{2} < \Vert D \Vert_{2} \rho$.
	Thus, (\ref{eq: shifting opnorm D from rho to lattice in nu-function}) follows.
	Now let $Q>0$, and $v\in\Gamma$ with $0<\left\Vert v\right\Vert _{2}\leq Q$.
	By the inequality of arithmetic and geometric mean, we have 
	\[
	\left\Vert Dv\right\Vert _{2}^{n}
	\geq n^{\nicefrac{n}{2}}\cdot\mathrm{Nm}(Dv)
	\underset{n}{\gg} \nu(\Gamma,Q^{\star}).
	\]
	Now suppose $\left\Vert v\right\Vert _{2}>Q$. Since
	$\left\Vert v\right\Vert_{2}=\left\Vert D^{-1}Dv\right\Vert_{2}
	\leq \left\Vert D^{-1}\right\Vert_{2}\left\Vert Dv\right\Vert_{2}$, 
	we conclude that
	\[
	\left\Vert Dv\right\Vert _{2}>\left\Vert D^{-1}\right\Vert_{2} ^{-1}Q.
	\]
	Hence, we have
	\[
	\left\Vert Dv\right\Vert _{2}\underset{n}{\gg}
	\min\bigl\{(\nu(\Gamma,Q^{\star}))^{\nicefrac{1}{n}},\,\left\Vert D^{-1}
	\right\Vert_{2} ^{-1}Q\bigr\}.
	\]
	Specialising $ Q\coloneqq\left\Vert D^{-1}\right\Vert_{2}$,
	and noticing that by the inequality of arithmetic and geometric mean, 
	$\nu(\Gamma,\gamma_{n})\underset{n}{\ll} 1$, we get (\ref{eq:  first minimum vs. nu of lattice}).
\end{proof}

\begin{lem} \label{lem: ergodic sum vs. nu and r} 
	Let $U$ be as in (\ref{eq: def. of U}), and let $\s\geq 1$.
	Then, we have 
	\[
	S(\Lambda^{\perp},\s)\underset{n}{\ll}
	\frac{\s^{n-1}}{\nu(\Gamma^{\perp},(2^{\s}\left\Vert U\right\Vert_{2})^{\star})}.
	\]
\end{lem}

\begin{proof}
	Since $\Lambda^{\perp}=U^{-1} \Gamma^{\perp}$,
	we conclude by (\ref{eq: first minimum vs. nu of lattice}) that
	\[
	S(\Lambda^{\perp},\s)=\sum_{\delta\in\Delta_{\s}}
	\frac{1}{\lambda_{1}^{n}(\delta U^{-1} \Gamma^{\perp})}
	\underset{n}{\ll}\sum_{\delta\in\Delta_{\s}}
	\frac{1}{\nu(\Gamma^{\perp},\left\Vert U \delta^{-1}\right\Vert_{2}^{\star})}.
	\]
	Since $\#\Delta_{\s} \underset{n}{\ll} \s^{n-1}$, and  since $\nu(\Gamma^{\perp},\cdot)$
	is non-increasing, we get 
	\[
	S(\Lambda^{\perp},\s)\underset{n}{\ll}
	\frac{\s^{n-1}}{\nu(\Gamma^{\perp},(2^{\s}\left\Vert U\right\Vert_{2})^{\star} )}.\qedhere
	\]
\end{proof} 

Now we can give the proof of Theorem \ref{thm: inhomogenous Skriganov}.
\begin{proof}[Proof of Theorem \ref{thm: inhomogenous Skriganov}]
	By (\ref{eq: shifting opnorm D from rho to lattice in nu-function}),
	we conclude
	\[
	r=n^{2}+\log\frac{\rho^{n}}{\nu(\Lambda^{\perp},\rho)}
	\leq n^{2}+\log\frac{\rho^{n}}{\nu(\Gamma^{\perp},\left\Vert U\right\Vert_{2} \rho)}=R
	\]
	Since $\nu(\Lambda^{\perp},\cdot)$ is non-increasing,
	and since $(2^{R}\left\Vert U\right\Vert_{2})^\star=2^{R}\left\Vert U\right\Vert_{2}$
	Lemma \ref{lem: ergodic sum vs. nu and r} yields 
	\begin{equation}\label{eq: ergodic sum vs. R and nu}
		S(\Lambda^{\perp},r)\underset{n}{\ll}
		\frac{R^{n-1}}{\nu(\Gamma^{\perp},2^{R}\left\Vert U\right\Vert_{2})}.
	\end{equation}
	By using Mahler's relation (\ref{eq: Mahler's relation})
	and Lemma \ref{lem: estimate nu function vs. n-th successive minimum}, 
	we obtain
	\begin{equation}\label{eq: n successive minimum vs. nu dual}
		\lambda_{n}^{n}(\Lambda)\underset{n}{\ll}
		\frac{1}{\lambda_{1}^{n}(U^{-1}\Gamma^{\perp})}
		\underset{n}{\ll}\frac{1}{\nu(\Gamma^{\perp},\left\Vert U\right\Vert_{2}^{\star} )}.
	\end{equation}
	Taking (\ref{eq: ergodic sum vs. R and nu}) and (\ref{eq: n successive minimum vs. nu dual})
	in (\ref{eq: homog. Skriganov app. to inhomog. box}) into account, it follows that
	\[
	\left|\#(\Gamma\cap B)-\mathrm{vol}(B)\right|
	\underset{n}{\ll}\frac{1}
	{\nu\bigl(\Gamma^{\perp},\left\Vert U\right\Vert_{2}^{\star} \bigr)}
	\biggl(\frac{\overline{t}^{n-1}}{\sqrt{\rho}}+
	\frac{R^{n-1}}{\nu(\Gamma^{\perp},2^{R}\left\Vert U\right\Vert_{2})}\biggr)
	\]
	which is (\ref{eq: inhomogenous Skriganov bound}).
\end{proof}

\section{Comparing $\nu(\Gamma,\cdot)$ and $\nu(\Gamma^{\perp},\cdot)$}\label{Section: comparing}
A natural question is whether one can state 
Theorem \ref{thm: inhomogenous Skriganov} in a way 
that is intrinsic in $\Gamma$, i.e. expressing $\nu(\Gamma^{\perp},\cdot)$
in terms of $\nu(\Gamma,\cdot)$. 
However, for $n>2$ there are weakly admissible lattices $\Gamma\subseteq\IR^{n}$ 
such that $\Gamma^{\perp}$ is not weakly admissible as the following example shows.

\begin{example}\label{ex: Gamma weakly adm. doesn t imply Gamma_perp weakly adm}
	Let $n\geq3$, and let $A'_0\in \GL_{n-1}(\IR)$ be	such that the
	elements of each row of $A'_0$ are $\mathbb{Q}$-linearly
	independent. Choose real $x_1,\ldots,x_{n-1},y$ outside of the
	$\IQ$-span of the entries of $A'_0$, and suppose $y\neq x_{n-1}$.
	Let $x=(x_1,\ldots,x_{n-1})^T$ and let $r_{n-1}$ be the last row of $A'_0$.
	Then, the matrix
	\[
	A_0\coloneqq\begin{pmatrix}A'_0 & x\\
	r_{n-1} & y
	\end{pmatrix}
	\]
	satisfies\medskip{}

	(i) $A_0\in\GL_{n}(\IR)$, and

	(ii) the elements in each row of $A_0$ are $\mathbb{Q}$-linearly independent.
	\medskip{}

	\noindent The second assertion is clear and for the first suppose 
	a linear combination of the rows vanishes. Using that the rows of $A'_0$ 
	are linearly independent over $\IR$	and that $y\neq x_{n-1}$, the first claim follows at once.
	We now let $A$ be the matrix we get from $A_0$ by swapping
	the first and the last row, and scaling each entry with 
	$|\det A_0|^{-1/n}$. Clearly,  (i) and (ii) remain valid for $A$,
	and the $(n,n)$-minor of $A$ vanishes.
	We conclude that $ \Gamma\coloneqq A\mathbb{Z}^{n} $
	is a unimodular, and weakly admissible lattice; moreover, Cramer's rule implies that 
	\begin{equation*}
		(A^{-1})^{T}=\begin{pmatrix}\star & \star & \ldots & \star\\
		\star & \ddots & \ddots & \vdots\\
		\vdots & \ddots & \star & \star\\
		\star & \ldots & \star & 0
		\end{pmatrix}
	\end{equation*}
	where an asterisk denotes some arbitrary real number, possibly a different
	number each time. Hence, $\Gamma^{\perp}$ contains a non-zero lattice point
	with a zero coordinate, and thus is not weakly admissible. 
\end{example}

Keeping Example \ref{ex: Gamma weakly adm. doesn t imply Gamma_perp weakly adm} in mind, 
we now concern ourselves with finding large subclasses of lattices $\Gamma\subseteq\IR^{n}$ such that
\begin{enumerate}
	\item \label{enu: wishlist for lattices vs. dual lattices 1}
	$\Gamma$ and $\Gamma^{\perp}$ are both weakly admissible, 
	\item \label{enu: wishlist for lattices vs. dual lattices 2}
	$\nu(\Gamma^{\perp},\cdot)=\nu(\Gamma,\cdot)$.
\end{enumerate}

It is easy to see that the first item holds for almost all lattices in the sense of the Haar-measure
on the space $\mathcal{L}_{n}=\SL_{n}(\IR)/\SL_{n}(\IZ)$ of unimodular lattices in $\IR^{n}$.
Moreover, we have the following criterion.

\begin{lem}\label{cor: weak admissib. of dual of algb. independet lattice}
	Suppose $A\in\SL_{n}(\IR)$, and suppose that the entries of $A$
	are algebraically independent (over $\IQ$).
	Then, $\Gamma\coloneqq A\mathbb{Z}^{n}$ and $\Gamma^{\perp}$
	are both weakly admissible. 
\end{lem}
\begin{proof}
	First note that if $K$ is a field and $X_1,\ldots,X_N$ are algebraically independent 
	over $K$, then any non-empty collection of pairwise distinct monomials $X_1^{a_1}\cdots X_N^{a_N}$ 
	is linearly independent over $K$.
	Next note that by Cramer's rule, each entry of $(A^{-1})^{T}$ is
	a sum of pairwise distinct monomials (up to sign) in the entries of $A$,
	and none of these monomials occurs in more than one entry of $(A^{-1})^{T}$. 
	This shows that the entries of $(A^{-1})^{T}$ are linearly independent over $\IQ$,
	in particular, the entries of any fixed row of $(A^{-1})^{T}$ are linearly independent over $\IQ$. 
	Thus, $\Gamma^\perp$ is weakly admissible.
	\end{proof}

Next, we prove Proposition \ref{lem: lattices where nu fncs. equal}.
	Notice that $\Sm$ and $\Sm^{-1}$ are, up to signs of the entries, 
	permutation matrices, and thus for every $w\in \IR^n$
	\begin{alignat}1\label{normSm}
	&\Nm(w)=\Nm(\Sm w)=\Nm(\Sm^{-1}w),\\
	&\label{distSm}	
	\left\Vert w\right\Vert _{2}=\left\Vert \Sm 
	w\right\Vert _{2}=\left\Vert \Sm^{-1}w\right\Vert _{2}.
	\end{alignat}
	Now let $Aw$ be an arbitrary lattice point in $\Gamma=A\mathbb{Z}^{n}$.
	Then, since $\Rm\in \IZ^{n\times n}$, we get $(A^{-1})^{T}\Rm w\in \Gamma^{\perp}$. 
	Since by hypothesis $A=\Sm^{-1}((A^{-1})^{T}\Rm)$, we conclude from (\ref{normSm}) 
	that $\Nm(Aw)=\Nm((A^{-1})^{T}\Rm w)$, and from (\ref{distSm}) 
	that $\left\Vert Aw\right\Vert _{2}=\left\Vert (A^{-1})^{T} \Rm w\right\Vert _{2}$.
	This shows that $\nu(\Gamma^{\perp}, \cdot)\leq \nu(\Gamma, \cdot)$.
	
	Similarly, if $(A^{-1})^{T} w\in \Gamma^{\perp}$ then, since $\Rm^{-1}\in \IZ^{n\times n}$,
	we find that $A\Rm^{-1}w\in \Gamma$, and using that $(A^{-1})^{T}=\Sm A\Rm^{-1}$
         we conclude as above that $\nu(\Gamma, \cdot)\leq \nu(\Gamma^{\perp}, \cdot)$.
	This proves Proposition \ref{lem: lattices where nu fncs. equal}.

\begin{rem}
	Let $I_{m}\coloneqq\mathrm{diag}(1,\ldots,1)$ be the identity
	matrix, and $0_{m}$ the null matrix in $\IR^{m\times m}$.
	Specialising
	\[
	\Sm=\Rm=\begin{pmatrix}0_{m} & I_{m}\\
	-I_{m} & 0_{m}
	\end{pmatrix}
	\]
	in Proposition \ref{lem: lattices where nu fncs. equal}, 
	we conclude that if $\Gamma=A\IZ^n$ with a symplectic matrix $A$, then 
	\begin{equation}\label{eq: nu dual equals nu original unimod lattice in dimension 2}
		\nu(\Gamma^{\perp},\cdot)=\nu(\Gamma,\cdot).
	\end{equation}
			Moreover, it is easy to see that
	$\mathrm{Sp}_{2}(\IR)=\SL_{2}(\IR)$, 
	and hence (\ref{eq: nu dual equals nu original unimod lattice in dimension 2}) holds 
	for any unimodular lattice $\Gamma\subseteq \IR^2$.
\end{rem}

Next, we prove Theorem \ref{prop: changing nu to nu dual }. 
Recall that $\alpha\coloneqq(\alpha_{1},\ldots,\alpha_{n})^{T}\in\IR^{n}$
is called badly approximable, if there is a constant $C=C(\alpha)>0$
such that for any integer $q\geq1$ the inequality
\begin{equation}
	\max\left\{ \left\Vert q\alpha_{1}\right\Vert ,\ldots,
	\left\Vert q\alpha_{n}\right\Vert \right\} 
	\geq\frac{C}{q^{\nicefrac{1}{n}}}\label{eq: badly approx vector dual formualtion}
\end{equation}
holds where $\left\Vert \cdot\right\Vert $ denotes the distance to
the nearest integer. By a well-known transference principle, cf. \cite{AItDA},
assertion (\ref{eq: badly approx vector dual formualtion}) is equivalent
to saying that for all non-zero vectors 
$q\coloneqq(q_{1},\ldots,q_{n})^{T}\in\mathbb{Z}^{n}$
the inequality 
\begin{equation}
	\left\Vert \left\langle \alpha,q\right\rangle \right\Vert 
	\geq\frac{\tilde{C}}{\left\Vert q\right\Vert _{2}^{n}}\label{eq: badly approximable vector}
\end{equation}
holds where $\tilde{C}=\tilde{C}(\alpha)>0$ is a constant. Let $\mathbf{Bad}(n)$ denote the set
of all badly approximable vectors in $\IR^{n}$. The crucial
step for constructing matrices generating the lattices announced in
Theorem \ref{prop: changing nu to nu dual }
is done by the following lemma.

\begin{lem}\label{lem: ex. of matrix with some prescribed alg. indept. entries}
	Let $n\geq3$ be an integer. Fix algebraically independent real numbers
	$c_{i,j}$ where $i,j=1,\ldots,n$ and $i\neq j$. Then, there exist
	$\lambda_{1},\ldots,\lambda_{n}\in\IR$ such that the entries
	of each row of
	\begin{equation}
		A\coloneqq\begin{pmatrix}\lambda_{1} & c_{1,2} & \ldots & c_{1,n}\\
		c_{2,1} & \lambda_{2} & \ddots & \vdots\\
		\vdots & \ddots & \ddots & c_{n-1,n}\\
		c_{n,1} & \ldots & c_{n,n-1} & \lambda_{n}
		\end{pmatrix}\label{eq: choice of matrix for nu vs. dual nu}
	\end{equation}
	are algebraically independent, $A$ is invertible,
	and each row-vector of $(A^{-1})^{T}$ is badly approximable.
\end{lem}

For proving this lemma, we shall use the following special case of
a recent Theorem of Beresnevich concerning badly approximable vectors.
We say that the map 
$F\coloneqq(f_{1},\ldots,f_{n})^{T}:\,\mathcal{B}\rightarrow\IR^{n}$,
where $\mathcal{B}\subsetneq \IR^{m}$ is a non-empty ball and $m,n\in\mathbb{N}$,
is non-degenerate, if $1,\,f_{1},\ldots,\,f_{n}$ are linearly
independent functions (over $\IR$). 

\begin{thm}[{\cite[Thm. 1]{Beresnevich15}}]
	\label{thm: badly approximable vectors and analytic maps}
	Let $n,\,m,\,k$ be positive integers. For each $j=1,\ldots,k$ 
	suppose that $F_{j}:\,\mathcal{B}\rightarrow\IR^{n}$
	is a non-degenerate, analytic map defined
	on a non-empty ball $\mathcal{B}\subsetneq\IR^{m}$. Then,
	\[
	\dim_{\mathrm{Haus}}\bigcap_{j=1}^{k} F_{j}^{-1}\left(\mathbf{Bad}(n)\right)=m.
	\]
\end{thm}

\begin{proof}[Proof of Lemma \ref{lem: ex. of matrix with some prescribed alg. indept. entries}]
	We work in two steps. First, we set the scene 
	to make use of Theorem \ref{thm: badly approximable vectors and analytic maps}.

	(i) Let $M\in\IR^{n\times n}$, and denote by $(M)_{i,j}$
	the entry in the $i$-th row and $j$-th column of $M$. Moreover,
	we define a map $\tilde{F}:\IR^{n}\rightarrow\IR^{n\times n}$
	by 
	\[
	\lambda\coloneqq(\lambda_{1},\ldots,\lambda_{n})^{T}\mapsto\begin{pmatrix}\lambda_{1} 
	& c_{1,2} & \ldots & c_{1,n}\\
	c_{2,1} & \lambda_{2} & \ddots & \vdots\\
	\vdots & \ddots & \ddots & c_{n-1,n}\\
	c_{1,n} & \ldots & c_{n,n-1} & \lambda_{n}
	\end{pmatrix}.
	\]
	On a sufficiently small non-empty ball $\mathcal{B}\subsetneq\IR^{n}$, centred at the origin,
	$\tilde{F}(\lambda)$ is invertible for every $\lambda\in \mathcal{B}$.\footnote{To see this,
	it suffices to show $\det\tilde{F}((0,\ldots,0)^T)\neq 0$. However, by the Leibniz formula,
	\[
	\det\tilde{F}(0,\ldots,0)=\sum_{\sigma}\mathrm{sgn}(\sigma)
	\prod_{i=1}^{n}c_{i,\sigma(i)}
	\]
	where the sum runs through all fixpoint-free permutations of $\{1, \ldots ,n\}$.
	Since $\left\{ c_{i,j}:\,i,j=1,\ldots,n,\,i\neq j\right\} $
	is algebraically independent, the evaluation of the polynomial on the right hand side above
	cannot vanish, cf. proof of 
	Lemma \ref{cor: weak admissib. of dual of algb. independet lattice}.} 
	On this ball $\mathcal{B}$, we define $F_{j}$, for $j=1,\ldots,n$,
	by mapping $\lambda$ to the $j$-th row of $(\bigl(\tilde{F}(\lambda)\bigr)^{-1})^T$.
	We claim that $F_{j}$ is a non-degenerate, and analytic
	map. By Cramer's rule, every entry of $((\tilde{F}(\lambda))^{-1})^T$
	is the quotient of polynomials in $\lambda_{1},\ldots,\lambda_{n}$
	whereas the polynomial in the denominator does not vanish on $\mathcal{B}$.
	Hence, each $F_{j}$ is an analytic function. Now we show
	that $F_{1}$ is non-degenerate, the argument for the other
	$F_{j}$ being similar. The $j$-th component of $F_{1}$
	is $(\bigl(\tilde{F}(\lambda)\bigr)^{-1}){}_{j,1}$
	and, using Cramer's rule, is hence of the shape
	\[
	(\det\tilde{F}(\lambda))^{-1} \Biggl(\mathcal{R}_{j}+(-1)^{1+j}
	\prod_{k=2,\,k\neq j}^{n}\lambda_{k}\Biggr)
	\]
	where the polynomial $\mathcal{R}_{j}\in\IR[\lambda_{2},\ldots,\lambda_{n}]$
	is of (total) degree $<n-1$, if $j=1$, and of (total) degree $<n-2$,
	if $j=2,\ldots,n$. Therefore, if a linear combination $k_{0}+
	\sum_{j=1}^{n}k_{j}((\tilde{F}(\lambda))^{-1})_{j,1}$ with scalars $k_{0},\ldots,k_{n}\in\IR$ 
	equals the zero-function $\mathbf{0}:\mathcal{B} \rightarrow \mathbb{R}$, then 
	\[
	\mathbf{0}= k_{0}\cdot (\det\tilde{F}(\lambda)) +\sum_{j=1}^{n}k_{j}(-1)^{1+j}
	\prod_{k=2,\,k\neq j}^{n}\lambda_{k}+\sum_{j=1}^{n}k_{j}\mathcal{R}_{j}.
	\]
	Comparing coefficients, we conclude that $k_{0}=0$ and thereafter
	$k_{1}=k_{2}=\cdots=k_{n}=0$. Hence, $F_{1}$ is non-degenerate.

	(ii) 
	By part (i), Theorem \ref{thm: badly approximable vectors and analytic maps}
	implies that the set $M$ of all $\lambda\in \mathcal{B}$ such that
	$F_{1}(\lambda),\ldots,F_{n}(\lambda)$ are all badly approximable,
	has full Hausdorff dimension. Moreover, we claim that there is a set
	$M^{(1)}\subseteq M$ of full Hausdorff dimension such that for every
	$\lambda\in M^{(1)}$ the entries of the first row of $\tilde{F}(\lambda)$ 
	are	algebraically independent. Let $M_{1}$ be the subset of $M$ of all
	elements $\lambda\coloneqq(\lambda_{1},\ldots,\lambda_{n})^{T}\in M$
	satisfying that $\left\{ \lambda_{1},c_{1,j}:\,j=2,\ldots,n\right\} $
	is algebraically \textit{dependent}; observe that the possible values for $\lambda_{1}$
	are countable, since $\mathbb{Z}[c_{1,2},\ldots,c_{1,n},x]$
	is countable and every complex, non-zero, univariate polynomial has only finitely many roots.
	Therefore, $M_{1}$ is contained in a countable	union of hyperplanes.
	It is well-known that if a sequence of sets $\{E_{i}\} \subseteq \mathbb{R}^{n}$ is given, then 
	$\dim_{\mathrm{Haus}}\bigcup_{i\geq1}E_{i}
	=\sup_{i\geq 1} \{\dim_{\mathrm{Haus}}E_{i}\}$, cf. \cite[p. 65]{BD}.
	Consequently,
	\[
	n=\dim_{\mathrm{Haus}}M=\max\left\{ \dim_{\mathrm{Haus}}(M\setminus M_{1}),
	\,\dim_{\mathrm{Haus}}M_{1}\right\} =\dim_{\mathrm{Haus}}(M\setminus M_{1}),
	\]
	and we define $M^{(1)}\coloneqq M\setminus M_{1}$. Using the same
	argument, we conclude that there is a set $M^{(2)}\subseteq M^{(1)}$
	of full Hausdorff dimension such that each of the first two rows of
	$\tilde{F}(\lambda)$ has algebraically independent entries
	for every $\lambda\in M^{(2)}$. Iterating this construction, 
	we infer that there is a subset $M^{(n)}\subseteq M^{(n-1)}\subseteq\ldots\subseteq M$
	of full Hausdorff dimension such that for every $\lambda\in M^{(n)}$ 
	each row of the matrix 	$A \coloneqq \tilde{F}(\lambda)$ 
	has algebraically independent entries, and 	$(A^{-1})^{T}$ has badly approximable row vectors. 
	Moreover, $\lambda\in M^{(n)}\subseteq \mathcal{B}$ implies that $A$ is invertible.
	
\end{proof}
We also need the following easy fact whose proof is left as an exercise.
\begin{lem}\label{lem: quotient lemma}
	Let $m\in \mathbb{N}$, and let $\alpha\in \IR$ be transcendental.
	Then, there are real numbers $\beta_{1},\ldots, \beta_{m}$ such that 
	$\beta_{1}, \alpha\beta_{1}, \beta_{2},\ldots, \beta_{m}$ 
	are algebraically independent.
\end{lem}

\begin{proof}[Proof of Theorem \ref{prop: changing nu to nu dual }.]
	First, we set $\tilde{\psi}(x)=\psi(x^2)$ such that for every $c>0$
	and $x\geq c$ we have $\tilde{\psi}(x)\leq\psi(cx)$.
       	We may assume that $\tilde{\psi}(q)\ll \exp(-q)$. 
	By writing down
	a suitable decimal expansion,
	we conclude that there exists a number $\alpha\in(0,1)$ such that
	\begin{equation}
		\Bigl|\alpha-\frac{p}{q}\Bigr|<\frac{\tilde{\psi}(q)}{q^{n+1}}
	\label{eq: approximations to generalized L numbers}
	\end{equation}
	has infinitely many coprime integer solutions $p,\,q\in\mathbb{Z}$;
	observe that such an $\alpha$ is necessarily transcendental.
	We apply Lemma \ref{lem: quotient lemma} with $m=n^2-n$ and we set 
	$c_{1,2} \coloneqq \beta_{1}, c_{1,3} \coloneqq \alpha \beta_{1}$, 
	and we choose exactly one value $\beta_k$ ($k\geq 2$) for each of the remaining $c_{i,j}$  ($i\neq j$).
	Thus, the real numbers $c_{i,j}$ are algebraically independent. 
	We use Lemma \ref{lem: ex. of matrix with some prescribed alg. indept. entries}
	with these specifications to find $A$ as in (\ref{eq: choice of matrix for nu vs. dual nu}).
	For $l\in \mathbb{N}$ let $p_{l},\,q_{l}$ denote  distinct 
	solutions to (\ref{eq: approximations to generalized L numbers}),
	and put $v_{l}\coloneqq(0,-p_{l},q_{l},0,\ldots,0)^{T}\in\IZ^{n}$.
	Set $\tilde{A}:=|\det A|^{-1/n}A$,
	and let us consider the unimodular, weakly admissible lattice 
	$\Gamma\coloneqq \tilde{A}\mathbb{Z}^{n}$.
	Then, the first coordinate of $\tilde{A}v_{l}$ equals
	\[
		|\det A|^{-1/n}\left|-p_{l}c_{1,2}+q_{l}c_{1,3}\right|
		=|\det A|^{-1/n}\left|c_{1,2}\right|\left|q_{l}\alpha-p_{l}\right|
		\underset{A}{\ll} \frac{\tilde{\psi}(q_{l})}{q_{l}^{n}}.
	\]
	Since $\alpha\in(0,1)$, we may assume, by choosing $l$ large enough,
	that $p_{l} \leq q_{l}$. Hence, the $j$-th coordinate for $j=2,\ldots,n$
	of $\tilde{A}v_{l}$ is $\underset{A}{\ll} q_{l}$. Thus, for $l$ sufficiently large,
	\[
		\mathrm{Nm}(\tilde{A}v_{l})\underset{A}{\ll}
		\frac{\tilde{\psi}(q_{l})}{q_{l}^{n}}\cdot q_{l}^{n-1}=\frac{\tilde{\psi}(q_{l})}{q_{l}}
		\leq \frac{\psi(2\Vert \tilde{A}\Vert_{2} q_{l})}{q_{l}}
		\leq \frac{\psi(\Vert \tilde{A}v_{l}\Vert_2)}{q_{l}}.
	\]
	Choosing $\rho_l=\Vert \tilde{A}v_{l}\Vert_2$, we conclude that
	$\nu(\Gamma,\rho_l)\leq \psi(\rho_l)$ for all $l$ sufficiently large.

	Because the rows of $(A^{-1})^{T}$ are badly approximable vectors by construction,
	$\Gamma^{\perp}$ is weakly admissible. Moreover, by (\ref{eq: badly approximable vector}),
	we conclude that $\mathrm{Nm}((A^{-1})^{T}v)\underset{A}{\gg} 
	\left\Vert v\right\Vert_{2}^{-n^2}$ for every non-zero
	$v\in \IZ^n$. Also note that $\left\Vert(A^{-1})^{T}v\right\Vert_2<\rho$
	implies $\left\Vert v\right\Vert_2 < \Vert A^{T}\Vert_{2}\rho$. This implies that
	$\nu(\Gamma^{\perp},\rho)\underset{A}{\gg} \rho^{-n^{2}}$. 
	Hence, $\Gamma$ has the desired properties.
\end{proof}

\section{An Application - Proof of Corollary \ref{thm: main counting bound}}
\label{Section: proof of asyptotic expansion of Dio. counting function}
Throughout this section we fix the unimodular lattice $\Gamma=A\mathbb{Z}^{2}$ where
\[
A:=\frac{1}{\sqrt{\alpha}}\begin{pmatrix}1 & \alpha\\
1 & 2\alpha
\end{pmatrix},
\]
and we consider the aligned box
\begin{equation} 
B\coloneqq\frac{1}{\sqrt{\alpha}}
\bigl(
	\bigl[y,\,y+\varepsilon    \bigr]
	\times\bigl[y,\,y+\alpha t \bigr]
\bigr).\label{eq: def. of box in application}
\end{equation}
Then, the following relation holds
\begin{equation}
\nonumber \#(B\cap\Gamma)=\#\biggl\{(p,q)\in \IZ^2:\begin{array}{c}
0\leq p+\alpha q-y\leq\varepsilon,\\
0\leq p+2\alpha q-y\leq\alpha t
\end{array}\biggr\}.
\end{equation}
Because of (\ref{eq: assumptions on varepsilon t and alpha}), we conclude that
\begin{equation}\label{eq: counting function is counting lattice points}
	|N_{\alpha,y}(\varepsilon,t)-\#(B\cap\Gamma)|\underset{\alpha}{\ll} 1.
\end{equation}

In order to use Theorem \ref{thm: inhomogenous Skriganov},
we need to control the characteristic quantity $\nu(\Gamma,\cdot)$
of the lattice $\Gamma$. This is where the Diophantine properties of $\alpha$
come into play.
\begin{lem}\label{lem: nu lattice in dim 2 vs. Diophantine type}
	Let $\phi$ be as in (\ref{eq: definition of phi as lower bound}),
	and suppose $\rho > \gamma_{2}^{\nicefrac{1}{2}}$. Then, we have 
	\[
	\nu(\Gamma^{\perp},\rho)=\nu(\Gamma,\rho)
	\geq\frac{\phi(4\rho/\sqrt{\alpha})}{4}.
	\]
\end{lem}

\begin{proof}
	The claimed equality follows immediately 
	from Proposition \ref{lem: lattices where nu fncs. equal}, and the remark thereafter.
	A vector $v\in\Gamma$ is of the shape
	$$v = \frac{1}{\sqrt{\alpha}} \begin{pmatrix} z\\ z' \end{pmatrix}$$ 
	where $z\coloneqq p+q\alpha$, $z' \coloneqq z+q\alpha$, and $p,q$ denote integers.
	Assume that $\left\Vert v\right\Vert_{2}\in(0,\rho)$.
	Observe that $q=0$ implies $\mathrm{Nm}(v)\geq 1> 4^{-1} \phi(4\rho/\sqrt{\alpha})$.
	Therefore, we may assume $q\neq 0$. Since $z'-z=q\alpha$,
	one of the numbers $\vert z \vert,\vert z' \vert$ 
	is at least $\frac{1}{2}\alpha \vert q\vert$, and both are bounded from below by 
	$\frac{1}{2\vert q\vert}\phi(2\vert q \vert)$. Hence,
	\[
	\Nm (v) \geq \frac{\alpha \vert q\vert}{2\sqrt{\alpha}}\cdot 
	\frac{\phi(2\vert q\vert)}{2\vert q\vert \sqrt{\alpha}} \geq \frac{\phi(4\rho/\sqrt{\alpha})}{4}
	\]
	where in the last step we used that
	$\frac{1}{2} \sqrt{\alpha} \vert q \vert 
	\leq\frac{1}{\sqrt{\alpha}} \min\{\vert z \vert, \vert z' \vert\}
	\leq \Vert v \Vert_{2} < \rho.$
\end{proof}
 
\begin{proof}[Proof of Corollary \ref{thm: main counting bound}]
	Let $B$ be given by (\ref{eq: def. of box in application}). 
	Thus, $B$ has sidelengths $t_{1}=\alpha^{-\nicefrac{1}{2}}\varepsilon$,
	and $t_{2}=\sqrt{\alpha}t$. By (\ref{Hermite}) and 
	(\ref{eq: assumptions on varepsilon t and alpha}), 
	we are entitled to take $\rho\coloneqq \varepsilon t > \gamma_{2}^{\nicefrac{1}{2}}$
	in Theorem \ref{thm: inhomogenous Skriganov}.
	Moreover, (\ref{eq: assumptions on varepsilon t and alpha}) implies
	$t_{1}< 1 <t_{2}$, and thus
	\[
	T=\sqrt{\alpha\frac{t}{\varepsilon}}
	> \sqrt{\varepsilon t} > 2 > \gamma_{2}.
	\]
	Hence, $T^{\star}=T$. By combining relation
	(\ref{eq: counting function is counting lattice points}) 
	and Theorem \ref{thm: inhomogenous Skriganov} with these specifications,
	it follows that 
	\begin{equation}\label{eq: first estimate for Diophantine counting fnc.}
		\vert N_{\alpha,y}(\varepsilon,t) - \varepsilon t \vert 
		\underset{\alpha}{\ll} \frac{1}{\nu(\Gamma^{\perp}, T)}  
		\left( 1 + \frac{R}{\nu(\Gamma^{\perp},2^{R}T)} \right).
	\end{equation}
	By Lemma \ref{lem: nu lattice in dim 2 vs. Diophantine type}, 
	the right hand side above is $ \ll R(\phi(4T/\sqrt{\alpha})
	\phi(2^{R+2}T/\sqrt{\alpha}))^{-1}.$
	The first factor in the round brackets is larger than the second one, 
	since $\phi$ is non-increasing. Hence, we conclude that the right hand-side of (\ref{eq: first estimate for Diophantine counting fnc.}) is bounded by 
	\begin{equation}\label{Appupperbound}
	\ll R(\phi(2^{R+2}T/\sqrt{\alpha}))^{-2}. 
	\end{equation}
	Furthermore, Lemma \ref{lem: nu lattice in dim 2 vs. Diophantine type} yields
	\begin{equation}\label{eq: estimate for R in app.}
		R\leq 4+\log\frac{4(\varepsilon t)^{2}}{\phi(4t\sqrt{\varepsilon t})}
		\ll \log\frac{\varepsilon t}{\phi(4t\sqrt{\varepsilon t})}.
	\end{equation}
	By using the first estimate from (\ref{eq: estimate for R in app.}), we get
	\[
	2^{R}\leq 2^{4}
	\left(\frac{4(\varepsilon t)^{2}}{\phi(4t\sqrt{\varepsilon t})}
	\right)^{\log2}
	< 2^{4+2\log2}\frac{(\varepsilon t)^{2}}{\phi(4t\sqrt{\varepsilon t})}.
	\]
	Hence, (\ref{Appupperbound}) is bounded from above by
	\[
	\ll \frac{\log\frac{\varepsilon t}{\phi(4t\sqrt{\varepsilon t})}}
	{\phi^{2}\left(2^{6+2\log2}
	\frac{(\varepsilon t)^{2}}{\phi(4t\sqrt{\varepsilon t})}
	\sqrt{\frac{t}{\varepsilon}}\right)} 
	\leq \frac{\log E}{\phi^{2} (E')}.
	\]
	This completes the proof of Corollary \ref{thm: main counting bound}.
\end{proof}

\section*{acknowledgements}
We would like to thank Carsten Elsner for sending us a preprint of his work.

\bibliographystyle{amsplain}
\bibliography{literature}
\end{document}